\documentclass[12pt]{amsart}
\usepackage{amsmath}
\usepackage{amssymb,amscd}
\usepackage[T1]{fontenc}
\usepackage[latin1]{inputenc}
\usepackage[colorlinks=true, urlcolor=blue,bookmarks=true,bookmarksopen=true, citecolor=blue,hypertex]{hyperref}

\numberwithin{equation}{section}

\newtheorem{defn}{Definition}[section]
\newtheorem{theorem}{Theorem}[section]

\newtheorem{corollary}[theorem]{Corollary}
\newtheorem{example}[theorem]{Example}

\newtheorem{lemma}[theorem]{Lemma}
\newtheorem{proposition}[theorem]{Proposition}
\newtheorem{remark}[theorem]{Remark}

\def \begineq{\begin{equation}}
\def \endeq{\end{equation}}

\def \bb{\mathbb}

\def \RR{{\bb{R}}}

\def \ZZ{{\bb{Z}}}

\def \({\left(}
\def \){\right)}
\def \<{\langle}
\def \>{\rangle}

\begin{document}
\title[A classification result  and Contact Structures in oriented compact cyclic 3-orbifolds]{A classification result and   Contact Structures in oriented cyclic 3-orbifolds}
\author[Saibal Ganguli]{Saibal Ganguli}
\address{ Harish-Chandra Research Institute
Chhatnag Road, Jhusi
Allahabad 211 019
India }
\email{saibalgan@gmail.com}

\subjclass[2010]{Primary 57R17 ,57R18, 57R55 Secondary 53D10 ,}
\keywords{ Contact structures, cyclic orbifolds, oriented,overtwisted structures ,Lutz twist}
\abstract
{ We  prove every oriented compact cyclic $3$-orbifold has a contact structure. There is another proof in the web by Daniel Herr in his uploaded thesis which depends on 
open book decompositions, ours is 
independent
 of that. We define overtwisted contact structures,  tight contact structures and Lutz twist on oriented compact  cyclic 3-orbifolds. We show every contact
structure in  an oriented compact cyclic $3$-orbifold contactified  by our method is homotopic to an overtwisted structure with the  overtwisted disc intersecting the 
singular locus of the orbifold. We
pose Eliashberg's like characterization of overtwisted contact structures of cyclic $3$-orbifolds as an open problem. In course of proving the above results we prove a classification
result for compact oriented cyclic-3 orbifolds which has not been seen by us in literature before.
}
\endabstract
\maketitle
\section{{\bf Introduction }}
 Contact structures are generally
known on manifolds. They are maximally non-integrable co-dimension one distributions.
 Contact structures can be also be defined on orbifolds by going to the level
of charts and keeping the distributions invariant under local group actions. Though
lot of work has been carried out on contact manifolds little is known about contact
 orbifolds. In this paper  we provide contact structures to compact oriented cyclic-3orbifolds.
 
       Cyclic orbifolds are orbifolds where the local group acting on an  orbifold chart are  cyclic and action on each chart in some  atlas is orientation preserving. The compact  
 $3$-dimensional version of these 
 orbifolds are topologically
 manifolds and we call them oriented if the manifold is orientable with an orientation. We prove Martinet-like theorem on these orbifolds. 
 
 The main idea of the proof
 is that since the singular locus of the orbifold is a link  we can get a rotationally invariant  contact structure transverse to the link by Eliashberg's extension theorem on the
 smooth 
 $3$-manifold (the unique
 smooth structure(up to diffeomorphism) it gets by virtue of it topologically being a $3$-manifold). Then we prove a classification result for these orbifolds. We prove any 
 such cyclic orbifold will
 be a quotient of  cyclic local  group actions on the $3-$ manifold which are global cyclic group actions along the neighborhoods of the singular link components.
   By virtue of this classification the
 rotationally-symmetric  manifold contact structure near the link induces a contact structure on the orbifold  structure. The above  classification has not been seen by
 us in literature before.
 
     Since the underlying topological space is a manifold  and since the contact structure on the orbifold structure induces a contact structure on the manifold structure 
     (in this case) 
     we can generalize
     overtwisted structures, tight structures and Lutz-twists to these orbifolds. This opens many questions solved in contact manifold theory to these orbifolds one of which is whether Eliashberg's
     like characterization of overtwisted contact structures is possible for compact  cyclic
     $3$-orbifolds.

     \section{Acknowledgement}
 I thank Professor Mainak Poddar for introducing this topic. I thank Professor Yakov Eliashberg, Professor Francisco Presas and Roger Casals(for the contact structure extension argument)
 for answering my mails. I thank
 Daniel Herr for uploading his thesis \cite{[DH]}  and giving me the motivation to carry out this work. I thank Harish Chandra Research Institute and institute of mathematical sciences
 Chennai for my post doctoral grant under which
 the project has been carried out.
     
     \section{{\bf Contact Structure}}
 We give a short exposition of contact structures and contact topology for a detailed exposition the reader may consult \cite{[HG]}.
 \begin{defn}
  A contact structure $\zeta$ in a $2n+1$ dimensional manifold is a $2n$ dimensional maximally non-integrable plane distribution. In dimension $3$ it means around 
  each point
  $p$ if we take 
  two  locally defined  linearly independent vector fields $X_p$ and $Y_p$  lying in $\zeta$ then their lie bracket  does not lie in $\zeta$. We call the  manifold $M$ %
   with a contact structure $\zeta$ a contact manifold $( M,\zeta)$. 
 \end{defn}
\begin{remark}
   When the distribution is co-orientable we can define a global form $\alpha$ such that $\alpha (d\alpha)^{n}$ is a no where zero  top degree form of the $2n+1$ dimensional 
   manifold and  kernel of
   $\alpha$ is $\zeta$. It can be argued from this $3$-manifolds admitting contact structure are orientable.
 \end{remark}

 \begin{defn}
  A knot in a 3-manifold $M$ is an embedding of a circle in the manifold.
 \end{defn}
 
 \begin{defn}
  A  knot $K$ in a contact $3$-manifold $(M,\zeta)$ is called Legendrian if it is tangent to the contact structure. It is called transverse if it is transverse to  the contact structure.
 \end{defn}

\begin{defn}
  A framing of a knot $K$ in a  oriented $3$-manifold is a framing of its normal bundle.
 \end{defn}

 \begin{defn}
  A link in a 3-manifold $M$ is a one dimensional sub-manifold with connected components as knots.
 \end{defn}

 \begin{defn}
 A   contact framing of a Legendrian knot $K$  is one in which one of the frame vector is tangent to the contact structure.
\end{defn}

\begin{defn}
 A   surface framing of a Legendrian knot $K$ with respect  to a surface $\Sigma$ whose boundary is $K$ is a framing where one of the frame vectors is tangent to the surface.
\end{defn}

\begin{defn}
  A characteristic foliation of a surface  $\Sigma$ in a contact $3$-manifold $(M,\zeta)$ is a singular foliation formed by the intersection of the tangent space of $\Sigma$ with $\zeta$.
\end{defn}

\begin{defn}
 An overtwisted disc $D$ is a disc in a contact $3$-manifold $(M,\zeta)$  with  Legendrian boundary $K$  such that the contact framing of $K$ does not twist with respect to 
 surface framing and the characteristic foliation has exactly one interior singular point.
\end{defn}

\begin{defn}
A  contact structure which has  a overtwisted disc is called a overtwisted contact structure. It is called tight if it is not overtwisted.
\end{defn}

\begin{example}\label{over}
 We give an example of an overtwisted contact structure $\zeta_{ot}$ in $\RR^{3}$. It can be found also in \cite{[HG]} section 4.5. We define the structure by the following equation in standard cylindrical coordinates
 \begin{equation}
  cos(r)dz +rsin(r)d\phi =0
 \end{equation}
 
Take the disc D$= \{ r| r \leq \pi ,z=0 \}$. The characteristic foliation is singular at $ 0$ and along the boundary. If we perturb this disc near the boundary , the characteristic 
foliation looses its
singularity at the boundary and the boundary  becomes a leaf of the foliation making it an  overtwisted disc  of  the structure $\zeta_{ot}$ and thus  $\zeta_{ot}$ is an overtwisted structure. 
We call the 
perturbed  or unperturbed disc $D$, $\Delta$ the standard overtwisted disc and $\zeta_{ot}$  standard overtwisted structure.
\end{example}
\subsection{{\bf Lutz twist}}
We consider  a oriented  $3$- manifold with a contact structure $\zeta$ and a transverse oriented knot then it is known in a tubular neighborhood of the knot   $\zeta$ is  the kernel
of the form $dz+r^{2}d\theta$ where $z$ is along the knot direction and $r$  and $\theta$ are  coordinates along meridian direction of the tubular neighborhood  which coincides 
with the normal bundle of the knot in this case.
 We say $\zeta^{'}$ is obtained from $\zeta$ with a Lutz twist if on $S^{1}\times D^{2}$ the new contact structure $\zeta $ is defined by
\begin{equation}
 \zeta^{'}=ker({h_1(r)dz+h_2(r)d\phi})
\end{equation}
and $\zeta{'}$ coincides with with $\zeta$ outside the solid torus. Conditions are
\begin{enumerate}
\item  $h_1(r)=-1,h_2(r)=-r^2 $near $r=0$.

\item $h_1(r)=1$ and $h_2(r)=r^2$ near $ r=1$.

\item  $(h_1(r),h_2(r))$ is never parallel to $({h_1}^{'},{h_2}^{'})$.
\end{enumerate}
\subsection{ {\bf Full Lutz twist}}
Changing condition 1 and   adding   some  more conditions we  get a full Lutz twist.
\begin{equation}
  h_1(r)=1
 \end{equation}
 \begin{equation}
 h_2(r)=r^2 
 \end{equation}
  for $ r \in [0,\epsilon] \cup [1-\epsilon,1] $.
 By Lemma 4.5.3 in \cite{[HG]},
 \begin{lemma}
  A full Lutz twist does not change the homotopy class of $\zeta$ as
 a $2$-plane field.
 \end{lemma}
It can be seen $h_2(r)$ is zero in two places (for details fig($4.11$) \cite{[HG]}).Taking $r_0$   the smaller radius where $h_2(r)=0$ consider the following embedding
 $ \phi:$ ${D^{2}}_{r_0} \rightarrow S^{1} \times D^{2} $.

 $(r,\phi) \rightarrow (z(r),r,\phi)$
 where $z(r)$ is a smooth function with $z(r_0)=0$ and $z(r) > 0$ for $0 \leq r < r_0 $ and $z^{'}(r)=0$ for only $r=0$ . It can be seen $\phi(\delta(D^{2}_{r_0}))$ is Legendrian and its 
 contact framing and surface framing
 do not twist with respect to each other and with some more conditions and  work we can show $ \phi(D^{2}_{r_0})$ is an overtwisted disk.(for more details section 4.5 \cite{[HG]})
\begin{proposition}
 We thus conclude that any contact structure is homotopic to an overtwisted contact structure  got by a full Lutz twist.
\end{proposition}
\section{\bf{ Orbifolds}}
We give a brief introduction to orbifolds following the treatment in \cite{[ALR]}
\begin{defn}
Let X be a topological space, and fix n$ \geq 0$.
\begin{enumerate}\label{orb}
\item  An n-dimensional orbifold chart on X is given by a connected open subset
$\tilde{U} \subset R^{n}$, a finite group G of smooth diffeomorphisims of  $\tilde{U}$, and a map $\phi:
\tilde{U}$  $\rightarrow $ X so that $\phi$  is  G-invariant and induces a homeomorphism of $\tilde{U}/ G$onto
an open subset $ U \subset X$.

\item An embedding $ \lambda (\tilde{U},G,\phi) \rightarrow  $( $\tilde{V},H,\psi) $between two such charts is a
smooth embedding $\lambda :\tilde{U} \rightarrow \tilde{V}$ with $ \psi \lambda=\phi $.

\item An orbifold atlas on X is a family   of such charts $\mathcal{U}=(\tilde{U},G,\phi)$, which
cover X and satisfy the compatibility condition: given any two charts $(\tilde{U},G,\phi)$ for $U=\phi(\tilde{U}) \subset X$ and
 $(\tilde{V},H,\psi)$ for $ V=\psi(\tilde{V} )\subset X$   and a point  $x \in$ $U \cap V$ there
exists an open neighborhood $W \subset U \cap V$ of x and  a chart $(\tilde{W} ,K,\mu) $ for W such that there are embeddings
$\lambda_1: (\tilde{W} ,K,\mu)\rightarrow (\tilde{U} ,G,\phi) $ and $\lambda_2 :(\tilde{W},K,\mu)\rightarrow (\tilde{V} ,H,\phi)$. Every embedding of charts induces
injective homomorphisms $\lambda:K \rightarrow G$.
\item An atlas $\mathcal{U}$ is said to refine another atlas $\Upsilon$ if for every chart in $\mathcal{U}$ there
exists an embedding into some chart of  $\Upsilon$. Two orbifold atlases are said to be
equivalent if they have a common refinement.
\end{enumerate}
\end{defn}

\begin{defn}
An effective orbifold X of dimension n is a paracompact Hausdorff
space $\bf{X}$  equipped with an equivalence class $[\mathcal{U}]$ of n-dimensional orbifold
atlases such that each local group acts effectively in local charts.
\end{defn}
{\bf Throughout this paper we  always assume that our orbifolds are
effective}.
\begin{defn}
 If the finite group actions on all the charts are free, then $\bf{X}$ is locally
Euclidean, hence a manifold. Points in an effective  orbifold where there are non-trivial stabilizers
 form the {\bf singular locus}.
\end{defn}

\begin{defn}
Let $ \bf{X}$ $= (X,\mathcal{U})$  and $ \bf{Y}$ $=(Y,\Upsilon) $ be orbifolds. A map $ f:X \rightarrow Y $ is said to be smooth if for any point  $x$ $\in X$ there are charts
$(\tilde{U},G,\phi)$ around $x$ and $(\tilde{V},G,\psi)$, around $f(x)$ with the property that f maps $U=\phi(\tilde{U})$ into $V=\psi(\tilde{U})$and can be lifted to a smooth map $ \tilde{f}:\tilde{U} \rightarrow \tilde{V}$ satisfying  $ \psi\tilde{f}=f\phi $
\end{defn}
\begin{defn}
The isotropy subgroup of a point $\tilde{x}$ in the chart $(\tilde{U},G,\phi)$ is $G_{\tilde{x}}=\{ g \in G \mid g \tilde{x} = \tilde{x} \}$. If $\tilde{y}$
is a point in the $G$-orbit of $\tilde{x}$ , then $G_{\tilde{y}}$ is conjugate to $G_{\tilde{x}}$ in $G$.  If $G$ is Abelian ,$ G_{\tilde{x}} =$ $ G_{\tilde{y}}$ and we
denote $G_{\tilde{x}}$  by $G_x$ where $ x= \phi(\tilde{x})$.  $G_x$ is called the isotropy group of $x$.
\end{defn}
\subsection{Tangent Bundle of an Orbifold}
We  describe a  tangent bundle for an
orbifold $\bf{X}$ $=(X,\mathcal{U})$. Given a chart $(\tilde{U},G,\phi)$ we consider a tangent bundle $\mathcal{T}\tilde{U}$. Since  G acts smoothly on $\tilde{U}$, hence it  acts
smoothly on $\mathcal{T} \tilde{U} $. Indeed  if ($\tilde{u},v$) a typical element of $\mathcal{T} \tilde{U}$ then $g(\tilde{u},v)=(g\tilde{u},dg_{\tilde{u}}(v))$(here $d$ stands for
the differential). Moreover the projection 
map $\mathcal{T}\tilde{U}\rightarrow \tilde{U}$ is equivariant from which we get a projection $p :\mathcal{T}\tilde{U}/G \rightarrow U$ by using the map $\phi$.  If $ x=\phi(\tilde{x})$
\begin{equation}
    p^{-1}(x)=\{G (\tilde{x},v) \subset \mathcal{T}\tilde{U}/G \}
 \end{equation}
    It can be easily seen this fiber is homeomorphic to $ \mathcal{T}_{\tilde{x}}\tilde{U}/G_{\tilde{x}}$ where as $G_{\tilde{x}}$ is  the isotropy subgroup
    of the G action at $\tilde{x}$. This means we have constructed locally a bundle like object where the fiber is not necessarily a vector space, but
rather a quotient of the form $\RR^n/G_{0}$ where $G_{0} \subset Gl_{n}(R)$.
  To construct the tangent bundle on an orbifold  $ \tilde{X}=( X,\mathcal{U})$,we simply need to glue together the bundles defined over the
charts . Our resulting space will be an orbifold, with an atlas $\mathcal{T}\mathcal{U}$ comprising local
charts $(\mathcal{T}\tilde{U},G,\phi)$  over $ \mathcal{T}U=\mathcal{T}\tilde{U}/G$ for each $ (\tilde{U},G,\phi)\in \mathcal{U} $: We observe that the gluing maps $\lambda_{12}=\lambda_{2}{\lambda_{1}}^{-1}$  are smooth, so we can
use the transition functions $ d\lambda_{12}:\mathcal{T}\lambda_{1}(\tilde{W}) \rightarrow \mathcal{T}\lambda_{2}(\tilde{W})$ to glue $\mathcal{T}\tilde{U}/G \rightarrow U$ to $\mathcal{T}\tilde{U}/H \rightarrow H$. In other words, we define the space $\mathcal{T}X$ as an identification space $\bigsqcup_{\tilde{U} \in \upsilon}( \mathcal{T}\tilde{U}/G)/\sim)$ where we give it the minimal topology that will
make the natural maps $\mathcal{T}\tilde{U}/G \rightarrow \mathcal{T}X$ homeomorphisms onto open subsets of
$\mathcal{T}X$. We summarize this in the next remark.
\begin{remark}
The tangent bundle of an n-dimensional orbifold $\bf{X}$ denoted
by $\mathcal{T}\bf
{X} = (\mathcal{T}$$X,\mathcal{T}\mathcal{U})$has the structure of a $2n$-dimensional orbifold. Moreover,
the natural projection$ p:\mathcal{T}X \rightarrow X$ defines a smooth map of orbifolds, with
fibers $p^{-1}(x) =\mathcal{T}_{\tilde{x}}\tilde{U}/G_{\tilde{x}}$.
\end{remark}
\subsection{\bf{Example: Teardrop orbifold}}\label{tear}
A teardrop orbifold is a topological sphere  with an orbifold singularity at  north pole. Consider the standard embedding of the sphere in $\RR^3$ with$ (0,0,1)$
as the north pole and  $(0,0,-1)$as the  south pole. We give two orbifold charts. Let $U =\{(x,y,z) \in  S^2\mid z < \frac{1}{3} \}$ and $V =\{(x,y,z) \in  S^2\mid z > \frac{-1}{3} \}$ 
are two open sets around the south pole and north pole respectively. We define a  chart $(\tilde{U},G,\phi)$ where $\tilde{U}=U$, $G$ is the trivial group and $\phi$  identity
map. Around north pole we take the open set $V$
  and define an orbifold chart $(\tilde{V},H,\psi)$ where $\tilde{V}=V$, $ H=\ZZ_3$ and $\psi  $ is a triple cover branched at the north pole. Now if $ z \in U \cap V$ we take
  neighborhood so small that the maps $\phi$ and $\psi$   cover these neighborhood evenly. Let us call this neighborhood $ W_z$ and with the chart  $ (W_z,I,id)$  and the maps 
  $\lambda_1$ and
$\lambda_2$ are mere inclusions. Thus we have an orbifold structure which is different from the manifold structure of the sphere and we call it a teardrop orbifold.
From the above description it is fairly easy to construct its tangent bundle.

\begin{defn}\label{metric}
 A metric in an orbifold $(X,\mathcal{U})$ with tangent bundle$(\mathcal{T}$$X,\mathcal{T}\mathcal{U})$ is an entity which lifts to a metric in any 
 orbifold chart,invariant under local group action
 and invariant under embedding of charts. Every effective orbifold  with an orbifold structure as  defined above can be given a metric by the partition of unity.
\end{defn}
\begin{remark}\label{metric_1}
 By taking exponential map in the chart we can have charts where the local groups are subgroup of the orthogonal group.
\end{remark}

\section{{\bf Cyclic orbifold}}
\begin{defn}\label{metric}
 A  cyclic orbifold is an effective orbifold where the local groups are cyclic and  there exists an atlas where they are orientation preserving. From the above remark the local groups are 
 finite  cyclic subgroups of $SO(n)$. 
\end{defn}
\begin{remark}\label{cyc}
 All  cyclic $3$-orbifolds are topologically manifolds from the above definition. To see this  at a singular point (points where there is a non trivial stabilizer) the action fixes an axis
 and rotates transverse $2-$planes in charts. Since a $2$-plane quotients to a topological manifold under such action the topological space is a manifold. Thus for the cyclic $3$-orbifold 
 $\bf{X}$ $=(X,\mathcal{U})$ the topological space $X$ is a $3$-manifold.
 \end{remark}

 \begin{defn}\label{orient}
An oriented compact cyclic $3$-orbifold $\bf{X}$  is a compact orbifold where the underlying topological space $X$ is oriented and compact.
\end{defn}
 
 \begin{lemma}\label{link}
 The singular locus $S$ of a of a compact   cyclic $3$-orbifold  $\bf{X}$   with topological space $X$ is a link.
\end{lemma}

\begin{proof}
 Since around a singular point the action is that of a cyclic group of $SO(3)$ the fixed point set   of the local action is an axis which is an one manifold. 
 Since the singular locus is a closed
 set and the topological space $X$ is compact the connected  components are knots. 
\end{proof}

\begin{remark}\label{order}
The order of the local group along a connected component of the singular link does not change as we can see from above it is locally constant.
\end{remark}
 
 \subsection{{\bf Example}}
 Tear drop orbifold described in \ref{tear} is a cyclic orbifold. In three dimensions take a solid torus  lying in a  $3-$manifold and replace it by a $ Z_n$ quotient where the group acts as 
 rotation in meridian disc directions. The  resulting space has a cyclic orbifold structure.

 \section{{\bf Contact orbifold}}
\begin{defn}\label{order}
 A contact orbifold is an orbifold where there is a local group invariant contact structure in each orbifold chart and which is invariant under  embedding of charts.
\end{defn}

\section{{\bf Compact oriented cyclic  orbifolds have  contact structure}}
We prove in this section that
\begin{theorem}\label{cont}
 Any compact oriented cyclic  3-orbifold has a contact structure.
\end{theorem}
To prove this we  are required to prove  a classification theorem for oriented cyclic $3$-orbifold 
\begin{theorem}\label{class}
 Every oriented $3$-cyclic orbifold can be reduced to a quotient of  local actions which are  global cyclic group actions along the neighborhoods of the singular link components
 and their actions are rotations along  discs transverse to the link with respect to a suitable metric and coordinates compatible with respect to orbifold coordinates.
\end{theorem}

\begin{proof}
 Take a compact oriented cyclic orbifold $\bf{X}$ with a topological space $X$  and singular link $S$. Take a point $x$ on the singular link. Take an orbifold chart  
 $\mathcal{U}=(\tilde{U_x},G,\phi)$ where $G$ is a finite cyclic subgroup of $SO_3$ and $\phi((\tilde{U_x})$ is an open set in $X$ around the point $x$. The $G$ action fixes  ${\phi}^{-1}(S)$
 and preserves a 2-plane  in tangent space of ${\phi}^{-1}(x)$ (which is a vector space and covers the orbifold tangent space of $x$ which is not a vector space)transverse to the vector along the lift of the singular link. If $x$ lies in the overlap of image of two charts 
  $\mathcal{U}=(\tilde{U_x},G,\phi)$
 and $\mathcal{V}=(\tilde{V_x},H,\psi)$ the tangent space of ${\phi}^{-1}(x)$ and ${\psi}^{-1}(x)$ are connected by $d(\lambda_2 {\lambda_1}^{-1})$ where
 $\lambda _1:\tilde{W_x} \rightarrow \tilde{U_x}$ and  $\lambda _2:\tilde{W_x} \rightarrow \tilde{V_x}$ are embeddings of an  another  chart  $\tilde{W_x}$ around $x$ to the above two charts
 and $d$ means the differential ( for definition of embedding of charts please refer to \ref{orb}. It should also be noted that these embeddings induce isomorphism of local groups since the 
 order of each
 group involved is same as the order of singularity along the link component is same). It can be seen from definition of  embeddings the transition function 
 $d({\lambda_2} {\lambda_1}^{-1})$ maps
 the direction along the singular link $S$ in one chart to the direction of the singular link $S$ in the other chart and the  $G$ invariant plane to the $H$ invariant plane. 
 If we take any other 
 embedding $\lambda_3$ and$\lambda_4$
  the two differentials  $d(\lambda_2 {\lambda_1}^{-1})$ and $d(\lambda_4 {\lambda_3}^{-1})$ will differ by a composition of $2k\pi/n$ rotation ( they are smooth hence differ by a group element
  and $n$ is the order of this connected 
 component
 of
 the link). So changing $\lambda_4$ suitably (i.e multiplying by a $2k\pi/n$ rotation or an element of the local group) we get   embeddings which give rise to transition functions of these link-transverse planes. Thus we can glue trivialization of $2$-planes
 into a bundle $B$ around the connected component of the singular link.
 \begin{remark}\label{trans}
  Since  there exists charts around  each point  in the  singular link component  which is a replica of part of the tangent space of  a point lying in the link component (since we are taking 
  exponential maps to 
  construct charts),  
  embeddings   are linear  maps in some coordinates   thus their differentials coincide with   actual maps 
   along the above  link 
  transverse planes. Thus the bundle transition functions are actually transition functions of the orbifold in the transverse plane direction near the link points.
   
 \end{remark}

    Since   around  each point on the singular link the total space of the above bundle $B$ gets coordinates from  a neighborhood  chart we get 
    coordinates of the the total space around
    the point that  are  compatible with local charts of the orbifold. Now covering the connected component of the link by such neighborhoods  in the total space of 
    $B$ by tube lemma like arguments
     we will
    get a tubular neighborhood of the  connected component of the link such that  the bundle coordinates  are  compatible with the coordinates of the orbifold  $\bf{X}$. Now since
    the coordinates of the orbifold structure induces a smooth manifold structure on the complement of the link $S$ and creating similar tubular neighborhoods 
    around other components of
    the link we get a smooth manifold structure by gluing the tubular neighborhoods with the complement of the link by pasting points away from the link.(the gluing is possible
    because the 
    chart branched cover and the covered 
    space near the link component 
    are  homeomorphic around the points of the link, and diffeomorphic away from the link component( both  homeomorphic to a  $D^{2}$ (disc) bundle over $S^{1}$ because action of cyclic groups
by rotation    on meridian discs quotients to discs)).
      
      Since the topological space resulting from gluing  is  homeomorphic to the $X$ and since $X$ is an oriented manifold the bundle $B$ is an oriented bundle.
      Thus $B$ can be trivialized. Take an  element $g$ of a local group and  such that  it acts as a $2\pi/n$ rotation in the respective chart  and since 
      \begin{equation}\label{sec}
      d(\lambda_2 {\lambda_1}^{-1})dg =d(\eta(g))d(\lambda_2 {\lambda_1}^{-1})
     \end{equation}
     (where $\eta$ is a  group isomorphism)  and since the bundle $B$ is orientable $d(\eta(g))$ acts as a $2\pi/n$  rotation on  the link-transverse planes on the tangent bundle of the
     other chart. To see the above let $v$ belong to a 
     fiber in a trivialization of $B$
      corresponding to the first chart. Then $d(g)v$ and $v$ 
      form
      a $2\pi/n$ sector in the fiber. Then $d(\lambda_2 {\lambda_1}^{-1})dgv$ and $d(\lambda_2 {\lambda_1}^{-1})v$ will  form a $2k\pi/n$ sector by the above equation \ref{sec}. Since the 
      transition map
      is one-one and $\eta(g)$ is group homomorphism we have
      $k=1$. So $d(\eta(g))$ will be a $2\pi/n$ or a
      $-2\pi/n$ rotation, but since the bundle is orientable $\eta(g)$ is a $2\pi/n$ rotation in the other chart. So
       the action of the local groups can be made global on 
      the tubular neighborhood around the connected component of the link. Giving a metric  by gluing the standard metric in the charts in the tubular neighborhood, 
      which is invariant 
      under the  action of this group 
      we get $\partial{r}$ and $\partial{\theta}$ coordinates where action of the the group elements are  just translations by constants  in the theta direction.
       \end{proof}
       Now we  give the proof of theorem \ref{cont}
       \begin{proof}
       We define a contact structure around
      the tubular neighborhood( derived in theorem \ref{class}) given by $dz + r^{2}d(\theta)$ where  $z$ direction is the link direction.
      
         Now do the same thing for all components of the link.  Taking  a disjoint ball    pull back on it  the standard overtwisted structure with the standard overtwisted
         disc lying in it. We can extend these $2$-plane fields(the local contact structure around link components and the disjoint ball) to  global $2$-plane fields and   we can get a 
         contact structure homotopic to the global plane field which extends the contact structure around the link  on the whole manifold  by using Eliashberg's extension
         theorem (see  theorem 3.1.1  \cite{[YE]}) by taking $A$ as the link ,$L$ as the empty set and  $K$ as the one point set in the statement of the theorem. This structure is a 
         contact structure in the orbifold coordinates (see theorem \ref{class})around the link (since the tubular neighborhood coordinates  are compatible with orbifold
         coordinates)  and is invariant under the action of the local groups which are global along these neighborhoods of the link component (this follows from the previous theorem).
         Since away 
         from the link
           the manifold and orbifold coordinates are compatible we have a contact structure on the whole orbifold $\bf{X}$.
         \end{proof}

\begin{remark}
 The need for a neighborhood around the link component with compatible orbifold co-ordinates  which extends to smooth manifold coordinate system is necessary because
 otherwise the structure $dz + r^{2}d(\theta)$ would have lifted to a singular form in orbifold coordinates($dz +r^{2/n}d(\theta/n))$, if we just go by definition without 
 the above adjustment. This happens because the orbifold
 co-ordinates are branched cover over  the transverse discs of  actual manifold co-ordinates(the branched cover along the transverse discs
 can be best described by the map $w \rightarrow w^{n} $, hence the singular form) 
\end{remark}
\begin{corollary}
 The orbifold contact structure induces a contact structure in the manifold.
\end{corollary}
\begin{proof}
 Since the orbifold co-ordinates around the link components can be extended to a smooth manifold coordinate system the orbifold contact structure is a manifold contact structure transverse and
 rotationally invariant near the link components.
\end{proof}

\begin{remark}
  The oriented condition is necessary for the space to have  a contact structure since  from the  above corollary it is clear that  an orbifold contact structure on 
  these orbifolds induces a  contact structure on the manifold structure thus making it an oriented manifold.
\end{remark}

\begin{remark}
  There are cyclic orbifolds whose topological space is non-oriented. For example take a non-oriented manifold. Take a small Euclidean  open neighborhood and take a smaller solid torus. Remove the interior
  and replace it with a solid torus quotient of  a  $Z_n$ action  by rotation on meridian discs. This is a cyclic orbifold with a non-oriented manifold as a topological space.
\end{remark}

\section{{\bf Overtwisted and tight structures}}
\begin{defn}
 An overtwisted contact structure in an oriented compact contact cyclic $3$-orbifold is a contact structure which  lifts to an overtwisted 
 structure in the corresponding manifold structure. It is a tight structure if it is not overtwisted.
\end{defn}

\begin{defn}
 A Lutz twist in an oriented compact contact cyclic $3$-orbifold  for knots disjoint from the singular link $S$ and knots lying in the singular link is a Lutz twist on  the contact
 structure induced in the manifold. Since the twist parameters are rotationally invariant, on   the  orbifold contact structure given by our method the result of a Lutz twist is again another 
 contact structure in orbifold level. 
  In the same way we can define
 full Lutz twist.
\end{defn}
\begin{theorem}
 Any contact structure induced by our method is homotopic to an overtwisted contact structure with the overtwisted disc intersecting the singular link $S$.
\end{theorem}
\begin{proof}
Inspecting Lemma 4.5.3 in \cite{[HG]} the homotopy between a contact structure and the structure resulting from a full  Lutz twist is rotation  invariant. So around a knot in the  
singular link our contact structure can be changed to an overtwisted contact structure
by a homotopy in the orbifold level. Since an overtwisted disc intersects the knot around which the twist is taken  some overtwisted disc will intersect the knot in the singular link.
\end{proof}
 
 We end with a question.
\subsection{Question}
Can Eliashberg's type classification  is possible   for overtwisted structures in compact oriented cyclic 3-orbifolds(For statement of classification see \cite{[YE]})?


\begin{thebibliography}{[Dieu]}



\bibitem{[ALR]} A. Adem, J. Leida and Y. Ruan: Orbifolds and stringy topology,  Cambridge
 Tracts in Mathematics, {\bf 171}, Cambridge University Press, Cambridge, 2007.

 \bibitem{[DH]} Daniel Herr Open books in contact 3  orbifolds Dissertations and Theses 
University of Massachusetts - Amherst ScholarWorks@UMass Amherst
 
\bibitem{[HG]}  Hansjorg Geiges   An introduction to Contact topology, Cambridge studies in Advanced mathematics.
 
 \bibitem{[YE]} Y.Eliashberg  Classification of overtwisted contact structures on 3-manifolds, Inventiones mathematicae  98, 623-637(1989)
 
 
 \end{thebibliography}
\end{document}